\documentclass[12pt, reqno]{amsart}
\usepackage{amsmath, amsthm, amscd, amsfonts, amssymb, graphicx, color}
\usepackage[bookmarksnumbered, colorlinks, plainpages]{hyperref}
\hypersetup{colorlinks=true,linkcolor=red, anchorcolor=green, citecolor=cyan, urlcolor=red, filecolor=magenta, pdftoolbar=true}

\newtheorem{theorem}{Theorem}[section]
\newtheorem{lemma}[theorem]{Lemma}

\newtheorem{corollary}[theorem]{Corollary}
\theoremstyle{definition}

\theoremstyle{remark}

\numberwithin{equation}{section}

\begin{document}
\setcounter{page}{1}

\title[Bleimann-Butzer-Hahn operators
defined by $(p,q)$-integers]{Some approximation results on Bleimann-Butzer-Hahn operators
defined by $(p,q)$-integers}

\author[M. Mursaleen, Md. Nasiruzzaman, Asif Khan, and Khursheed J. Ansari]{M. Mursaleen$^1$, Md. Nasiruzzaman$^1$, Asif Khan$^2$ and Khursheed J. Ansari$^2$$^{*}$}

\address{$^{1}$ Department of Mathematics, Aligarh Muslim University, Aligarh--202002, India.}
\email{\textcolor[rgb]{0.00,0.00,0.84}{mursaleenm@gmail.com;
nasir3489@gmail.com}}

\address{$^{2}$ Department of Pure Mathematics, Aligarh Muslim University, Aligarh--202002, India.}
\email{\textcolor[rgb]{0.00,0.00,0.84}{asifjnu07@gmail.com; ansari.jkhursheed@gmail.com}}




\date{Received: xxxxxx; Revised: yyyyyy; Accepted: zzzzzz.
\newline \indent $^{*}$ Corresponding author}

\begin{abstract}
In this paper, we introduce a generalization
of the Bleimann-Butzer-Hahn operators based on $(p,q)$-integers and obtain
Korovkin's type approximation theorem for these operators. Furthermore, we
compute convergence of these operators by using the modulus of continuity.
\end{abstract} \maketitle

\section{Introduction and preliminaries}
Bleimann, Butzer and Hahn (BBH) introduced the following operators in \cite%
{brns} as follows;
\begin{align}\label{nas1}
L_n (f;x)= \frac{1}{(1+x)^n}\sum_{k=0}^n f\left(\frac{k}{n-k+1} \right)\left[
\begin{array}{c}
n \\
k%
\end{array}%
\right] x^k, x\geq 0
\end{align}
\parindent=8mmIn approximation theory, $q$-type generalization of Bernstein
polynomials was introduced by Lupa‏ \cite{lups}. In 1997, Phillips \cite%
{philip} introduced another modification of Bernstein polynomials. Also he
obtained the rate of convergence and the Voronovskaja's type asymptotic
expansion for these polynomials.

The BBH-type operators based on $q$-integers are defined as follows
\begin{align}\label{nas2}
L_{n}^{q}(f;x)=\frac{1}{\ell _{n}(x)}\sum_{k=0}^{n}f\left( \frac{[k]_{q}}{%
[n-k+1]_{q}q^{k}}\right) q^{\frac{k(k-1)}{2}}\left[
\begin{array}{c}
n \\
k%
\end{array}%
\right] _{q}x^{k}
\end{align}
where $\ell _{n}(x)=\prod_{k=0}^{n-1}(1+q^{s}x)$.\newline
Recently, Mursaleen et al \cite{mur7} applied $(p,q)$-calculus in
approximation theory and introduced first $(p,q)$-analogue of Bernstein
operators. They also introduced and studied approximation properties of $%
(p,q)$-analogue of Bernstein-Stancu operators in \cite{mur8}.\newline

Let us recall certain notations on $(p,q)$-calculus.

The $(p,q)$ integers $[n]_{p,q}$ are defined by
\begin{equation*}
[n]_{p,q}=\frac{p^n-q^n}{p-q},~~~~~~~n=0,1,2,\cdots, ~~0<q<p\leq 1.
\end{equation*}
whereas $q$-integers are given by
\begin{equation*}
[n]_{q}=\frac{1-q^n}{1-q},~~~~~~~n=0,1,2,\cdots, ~~0<q< 1.
\end{equation*}%
\newline
It is very clear that $q$-integers and $(p,q)$-integers are different, that
is we cannot obtain $(p,q)$ integers just by replacing $q$ by $\frac{q}{p}$
in the definition of $q$-integers but if we put $p = 1$ in definition of $%
(p,q)$ integers then $q$-integers becomes a particular case of $(p,q)$
integers. Thus we can say that $(p,q)$-calculus can be taken as a
generalization of $q$-calculus.\newline

Now by some simple calculation and induction on $n,$ we have $(p,q)$%
-binomial expansion as follows
\begin{equation*}
(ax+by)_{p,q}^{n}:=\sum\limits_{k=0}^{n}p^{\frac{(n-k)(n-k-1)}{2}}q^{\frac{%
k(k-1)}{2}} \left[
\begin{array}{c}
n \\
k%
\end{array}%
\right] _{p,q}a^{n-k}b^{k}x^{n-k}y^{k},
\end{equation*}
\begin{equation*}
(x+y)_{p,q}^{n}=(x+y)(px+qy)(p^2x+q^2y)\cdots (p^{n-1}x+q^{n-1}y),
\end{equation*}
\begin{equation*}
(1-x)_{p,q}^{n}=(1-x)(p-qx)(p^2-q^2x)\cdots (p^{n-1}-q^{n-1}x)
\end{equation*}%
\newline
and the $(p,q)$-binomial coefficients are defined by
\begin{equation*}
\left[
\begin{array}{c}
n \\
k%
\end{array}%
\right] _{p,q}=\frac{[n]_{p,q}!}{[k]_{p,q}![n-k]_{p,q}!}.
\end{equation*}

Again it can be easily verified that $(p,q)$-binomial expansion is different
from $q$-binomial expansion and is not a replacement of $q$ by $\frac{q}{p}$.%
\newline
By some simple calculation, we have the following relation

\begin{equation*}
q^k[n-k+1]_{p,q}=[n+1]_{p,q}-p^{n-k+1}[k]_{p,q}.
\end{equation*}

For details on $q$-calculus and $(p,q)$-calculus, one can refer \cite{vp},
\cite{mah,sad,vivek}, respectively.\newline

Now based on $(p,q)$-integers, we construct $(p,q)$-analogue of
BBH-operators, and we call it as $(p,q)$-Bleimann-Butzer-Hahn-Operators and
investigate its Korovokin's-type approximation properties, by using the test
functions $\left(\frac{t}{1+t}\right)^\nu$ for $\nu=0,1,2$. Also for a space
of generalized Lipschitz-type maximal functions we give a pointwise
estimation.

Let $C_{B}(\mathbb{R}_+)$ be the set of all bounded and continuous functions
on $\mathbb{R}_+$, then $C_{B}(\mathbb{R}_+)$ is linear normed space with
\begin{equation*}
\parallel f \parallel_{C_{B}}= \sup_{x \geq 0} \mid f(x) \mid.
\end{equation*}
Let $\omega$ denotes modulus of continuity satisfying the following
condition:
\begin{enumerate}
\item $\omega$ is a non-negative increasing function on $\mathbb{R}%
_+$

\item $\omega(\delta_1+\delta_2)\leq
\omega(\delta_1)+\omega(\delta_2)$

\item $\lim_{\delta \to 0}\omega(\delta)=0$.
\end{enumerate}
Let ${H}_\omega$ be the space of all real-valued functions $f$ defined on
the semiaxis $\mathbb{R}_+$ satisfying the condition
\begin{equation*}
\mid f(x) - f(y) \mid \leq \omega \left(\bigg{|} \frac{x}{1+x}- \frac{y}{1+y}
\bigg{|} \right),
\end{equation*}
for any $x,y \in \mathbb{R}_+$.
\begin{theorem}\label{main}\protect\cite{butz}
Let $\{A_n\}$ be the sequence of positive linear operators from $H_\omega$
into $C_B(\mathbb{R}_+)$, satisfying the conditions
\begin{equation*}
\lim_{n \to \infty} \parallel A_n \left( \left( \frac{t}{1+t}%
\right)^\nu;x\right)-\left(\frac{x}{1+x}\right)^\nu \parallel_{C_{B}},
\end{equation*}
for $\nu=0,1,2$. Then for any function $f \in H_\omega$
\begin{equation*}
\lim_{n \to \infty} \parallel A_n (f)-f \parallel_{C_{B}}=0.
\end{equation*}
\end{theorem}
We define $(p,q)$-Bleimann-Butzer and Hahn-type operators based on $(p,q)$%
-integers as follows:
\begin{align}\label{nas3}
L_n^{p,q}(f;x)=\frac{1}{\ell_n^{p,q}(x)}\sum_{k=0}^n f \left( \frac{
p^{n-k+1}[k]_{p,q}}{[n-k+1]_{p,q}q^k }\right) p^{\frac{(n-k)(n-k-1)}{2}}q^{%
\frac{k(k-1)}{2}} \left[
\begin{array}{c}
n \\
k%
\end{array}%
\right] _{p,q} x^k
\end{align}
where, ~~$x \geq 0,~~0< q<p\leq 1$
\begin{equation*}
\ell_n^{p,q}(x)= \prod_{s=0}^{n-1}(p^s+q^s x)
\end{equation*}
and $f$ is defined on semiaxis $\mathbb{R}_+$.\newline
And also by induction, we construct the Euler identity based on $(p,q)$%
-analogue defined as follows:
\begin{align}\label{nas4}
\prod_{s=0}^{n-1}(p^s+q^s x) =\sum_{k=0}^n p^{\frac{(n-k)(n-k-1)}{2}}q^{%
\frac{k(k-1)}{2}} \left[
\begin{array}{c}
n \\
k%
\end{array}%
\right] _{p,q} x^k
\end{align}

\parindent=8mmIf we put $p=1$, then we obtain $q$-BBH-operators.
If we take $f\left( \frac{[k]_{p,q}}{[n-k+1]_{p,q}}\right) $ instead of $%
f\left( \frac{p^{n-k+1}[k]_{p,q}}{[n-k+1]_{p,q}q^{k}}\right) $ in \eqref{nas3},
then we obtain usual generalization of Bleimann, Butzer and Hahn operators
based on $(p,q)$-integers, then in this case it is impossible to obtain
explicit expressions for the monomials $t^{\nu }$ and $\left( \frac{t}{1+t}%
\right) ^{\nu }$ for $\nu =1,2$. But if we define the Bleimann, Butzer and
Hahn operators as in \eqref{nas3}, then we can obtain explicit formulas for the
monomials $\left( \frac{t}{1+t}\right) ^{\nu }$ for $\nu =0,1,2$. We
emphasize that these operators are more flexible than the classical
BBH-operators and $q$-analogue of BBH-operators. That is depending on the
selection of $(p,q)$-integers, the rate of convergence of $(p,q)$%
-BBH-operators is atleast as good as the classical one.

\section{Main results}

\begin{lemma}\label{main1}
Let $L_n^{p,q}(f;x)$ be given by \eqref{nas3}, then for any $x \geq 0$ and $%
0<q<p\leq 1$ we have the following identities

\begin{enumerate}
\item $L_n^{p,q}(1;x)=1,$

\item\label{main2} $L_n^{p,q}(\frac{t}{1+t};x)=\frac{p[n]_{p,q}}{[n+1]_{p,q}}%
\left(\frac{x}{1+x}\right),$

\item \label{main4}$L_n^{p,q}\left((\frac{t}{1+t})^2;x\right)=\frac{%
pq^2[n]_{p,q}[n-1]_{p,q}}{[n+1]_{p,q}^2}\frac{x^2}{(1+x)(p+qx)} +\frac{%
p^{n+1}[n]_{p,q}}{[n+1]_{p,q}^2}\left(\frac{x}{1+x}\right).$
\end{enumerate}
\end{lemma}

\begin{proof}
\begin{enumerate}
\item $%
L_{n}^{p,q}\left( 1;x\right) =\frac{1}{\ell _{n}^{p,q}(x)}\sum_{k=1}^{n}%
 p^{\frac{(n-k)(n-k-1)}{2}}q^{\frac{k(k-1)}{2}}\left[
\begin{array}{c}
n \\
k%
\end{array}%
\right] _{p,q}x^{k}$ \\ \newline

For $0<q<p\leq 1$, we have
\begin{equation*}
\sum_{k=0}^{n}p^{\frac{(n-k)(n-k-1)}{2}}q^{\frac{k(k-1)}{2}}\left[
\begin{array}{c}
n \\
k%
\end{array}%
\right] _{p,q}x^{k}=\prod_{s=0}^{n-1}(p^{s}+q^{s}x)=\ell_{n}^{p,q}(x),
\end{equation*}
which completes the proof.\\ \newline

\item Let $t=\frac{p^{n-k+1}[k]_{p,q}}{[n-k+1]_{p,q}q^{k}}$, then $%
\frac{t}{t+1}=\frac{[k]_{p,q}p^{n+1-k}}{[n+1]_{p,q}}$
\begin{eqnarray*}
L_{n}^{p,q}\left( \frac{t}{1+t};x\right)  &=&\frac{1}{\ell _{n}^{p,q}(x)}%
\sum_{k=1}^{n}\frac{[k]_{p,q}p^{n-k+1}}{[n+1]_{p,q}}p^{\frac{(n-k)(n-k-1)}{2}%
}q^{\frac{k(k-1)}{2}}\left[
\begin{array}{c}
n \\
k%
\end{array}%
\right] _{p,q}x^{k} \\
&=&\frac{1}{\ell _{n}^{p,q}(x)}\sum_{k=1}^{n}\frac{[n]_{p,q}p^{n-k+1}}{%
[n+1]_{p,q}}p^{\frac{(n-k)(n-k-1)}{2}}q^{\frac{k(k-1)}{2}}\left[
\begin{array}{c}
n-1 \\
k-1%
\end{array}%
\right] _{p,q}x^{k} \\
&=&x\left( \frac{1}{\ell _{n}^{p,q}(x)}\cdot \frac{\lbrack n]_{p,q}}{%
[n+1]_{p,q}}p\right) \sum_{k=0}^{n-1}p^{\frac{(n-k)(n-k-1)}{2}}q^{\frac{%
k(k-1)}{2}}\left[
\begin{array}{c}
n-1 \\
k%
\end{array}%
\right] _{p,q}(qx)^{k} \\
&=&p\frac{[n]_{p,q}}{[n+1]_{p,q}}\frac{x}{1+x}.
\end{eqnarray*}

\item $L_n^{p,q}\left(\frac{t^2}{(1+t)^2};x\right) =\frac{1}{%
\ell_n^{p,q}(x)}\sum_{k=1}^n \frac{[k]_{p,q}^2p^{2(n-k+1)}}{[n+1]_{p,q}^2}
p^{\frac{(n-k)(n-k-1)}{2}}q^{\frac{k(k-1)}{2}} \left[
\begin{array}{c}
n \\
k%
\end{array}%
\right] _{p,q} x^k$.\newline

By some simple calculation, we have
\begin{equation*}
[k]_{p,q}=p^{k-1}+q[k-1]_{p,q},~~\mbox{and}~~
[k]_{p,q}^2=q[k]_{p,q}[k-1]_{p,q}+p^{k-1}[k]_{p,q},
\end{equation*}
using it we, get
\begin{equation*}
L_n^{p,q}\left(\frac{t^2}{(1+t)^2};x\right) =\frac{1}{\ell_n^{p,q}(x)}%
\sum_{k=2}^n \frac{q[k]_{p,q}[k-1]_{p,q}p^{2n-2k+2}}{[n+1]_{p,q}^2} p^{\frac{%
(n-k)(n-k-1)}{2}}q^{\frac{k(k-1)}{2}} \left[
\begin{array}{c}
n \\
k%
\end{array}%
\right] _{p,q} x^k
\end{equation*}
\begin{equation*}
+\frac{1}{\ell_n^{p,q}(x)}\sum_{k=1}^n p^{k-1}\frac{[k]_{p,q}p^{2n-2k+2}}{%
[n+1]_{p,q}^2} p^{\frac{(n-k)(n-k-1)}{2}}q^{\frac{k(k-1)}{2}} \left[
\begin{array}{c}
n \\
k%
\end{array}%
\right] _{p,q} x^k
\end{equation*}
\begin{equation*}
=\frac{1}{\ell_n^{p,q}(x)} \frac{q[n]_{p,q}[n-1]_{p,q}}{[n+1]_{p,q}^2}
\sum_{k=2}^n p^{\left((2n-2k+2)+\frac{(n-k)(n-k-1)}{2}\right)}q^{\frac{k(k-1)%
}{2}} \left[
\begin{array}{c}
n-2 \\
k-2%
\end{array}%
\right] _{p,q} x^k
\end{equation*}
\begin{equation*}
+\frac{1}{\ell_n^{p,q}(x)} \frac{[n]_{p,q}}{[n+1]_{p,q}^2}\sum_{k=1}^n
p^{\left((k-1)+(2n-2k+2)+\frac{(n-k)(n-k-1)}{2}\right)}q^{\frac{k(k-1)}{2}} %
\left[
\begin{array}{c}
n-1 \\
k-1%
\end{array}%
\right] _{p,q} x^k
\end{equation*}
\begin{equation*}
=x^2\frac{1}{\ell_n^{p,q}(x)} \frac{q[n]_{p,q}[n-1]_{p,q}}{[n+1]_{p,q}^2}
\sum_{k=0}^{n-2} p^{\left((2n-2k-2)+\frac{(n-k-2)(n-k-3)}{2}\right)}q^{\frac{%
(k+1)(k+2)}{2}} \left[
\begin{array}{c}
n-2 \\
k%
\end{array}%
\right] _{p,q} x^k
\end{equation*}
\begin{equation*}
+x\frac{1}{\ell_n^{p,q}(x)} \frac{[n]_{p,q}}{[n+1]_{p,q}^2}\sum_{k=0}^{n-1}
p^{\left(k+(2n-2k)+\frac{(n-k-1)(n-k-2)}{2}\right)}q^{\frac{k(k+1)}{2}} %
\left[
\begin{array}{c}
n-1 \\
k%
\end{array}%
\right] _{p,q} x^k
\end{equation*}
\begin{equation*}
=x^2\frac{1}{\ell_n^{p,q}(x)} \frac{pq^2[n]_{p,q}[n-1]_{p,q}}{[n+1]_{p,q}^2}
\sum_{k=0}^{n-2} p^{\frac{(n-k)(n-k-1)}{2}}q^{\frac{k(k-1)}{2}} \left[
\begin{array}{c}
n-2 \\
k%
\end{array}%
\right] _{p,q} (q^2x)^k
\end{equation*}
\begin{equation*}
+x\frac{1}{\ell_n^{p,q}(x)} \frac{p^{n+1}[n]_{p,q}}{[n+1]_{p,q}^2}%
\sum_{k=0}^{n-1} p^{\frac{(n-k)(n-k-1)}{2}}q^{\frac{k(k-1)}{2}} \left[
\begin{array}{c}
n-1 \\
k%
\end{array}%
\right] _{p,q} (qx)^k
\end{equation*}
\begin{equation*}
=\frac{pq^2[n]_{p,q}[n-1]_{p,q}}{[n+1]_{p,q}^2}\frac{x^2}{(1+x)(p+qx)} +%
\frac{p^{n+1}[n]_{p,q}}{[n+1]_{p,q}^2}\left(\frac{x}{1+x}\right).
\end{equation*}
\end{enumerate}
\end{proof}
\textbf{\ Korovkin type approximation properties.}\\

In this section, we obtain the Korovkin's type approximation properties for
our operators defined by \eqref{nas3}, with the help of Theorem \ref{main}.\newline

\parindent=8mmIn order to obtain the convergence results for the operators $%
L_{n}^{p,q}$, we take $q=q_{n},~~p=p_{n}$ where $q_{n}\in (0,1)$ and $%
p_{n}\in (q_{n},1]$ satisfying,
\begin{align}\label{nas5}
\lim_{n}p_{n}=1,~~~~~~\lim_{n}q_{n}=1
\end{align}

\begin{theorem}
Let $p=p_n,~~q=q_n$ satisfying \eqref{nas5}, for $0<q_n<p_n\leq 1$ and if $%
L_n^{p_n,q_n}$ is defined by \eqref{nas3}, then for any function $f \in H_\omega$,
\begin{equation*}
\lim_n \parallel L_n^{p_n,q_n}(f;x)-f \parallel_{C_{B}}=0.
\end{equation*}
\end{theorem}

\begin{proof}
Using the Theorem \ref{main}, we see that it is sufficient to verify that following
three conditions:
\begin{align}\label{nas6}
\lim_{n\rightarrow \infty }\parallel L_{n}^{p_{n},q_{n}}\left( \left( \frac{t%
}{1+t}\right) ^{\nu };x\right) -\left( \frac{x}{1+x}\right) ^{\nu }\parallel
_{C_{B}}=0,~~\nu =0,1,2
\end{align}%
From Lemma \ref{main1}, the first condition of \eqref{nas6} is fulfilled for $\nu =0$. Now
it is easy to see that from (\ref{main2}) of Lemma \ref{main1} \newline
\begin{eqnarray*}
\parallel L_{n}^{p_{n},q_{n}}\left( \left( \frac{t}{1+t}\right) ^{\nu
};x\right) -\left( \frac{x}{1+x}\right) ^{\nu }\parallel _{C_{B}} &\leq &%
\bigg{|}\frac{p_{n}[n]_{p_{n},q_{n}}}{[n+1]_{p_{n},q_{n}}}-1\bigg{|} \\
&\leq &\bigg{|}\left( \frac{p_{n}}{q_{n}}\right) \left( 1-p_{n}^{n}\frac{1}{%
[n+1]_{p_{n},q_{n}}}\right) -1\bigg{|}.
\end{eqnarray*}%
Since $%
q_{n}[n]_{p_{n},q_{n}}=[n+1]_{p_{n},q_{n}}-p_{n}^{n},~~[n+1]_{p_{n},q_{n}}%
\rightarrow \infty $ as $n\rightarrow \infty $, the condition \eqref{nas6} holds
for $\nu =1$. To verify this condition for $\nu =2$, consider (\ref{main4}) of Lemma
\ref{main1}. Then, we see that\newline
\newline
$\parallel L_{n}^{p_{n},q_{n}}\left( \left( \frac{t}{1+t}\right)
^{2};x\right) -\left( \frac{x}{1+x}\right) ^{2}\parallel
_{C_{B}}$ \\ \newline
$=\sup_{x\geq 0}\left\{ \frac{x^{2}}{(1+x)^{2}}\left( \frac{%
p_{n}q_{n}^{2}[n]_{p_{n},q_{n}}[n-1]_{p_{n},q_{n}}}{[n+1]_{p_{n},q_{n}}^{2}}.%
\frac{1+x}{p_{n}+q_{n}x}-1\right) +\frac{p_{n}^{n+1}[n]_{p_{n},q_{n}}}{%
[n+1]_{p_{n},q_{n}}^{2}}\cdot \frac{x}{1+x}\right\} $.\newline
A small calculation leads to
\begin{equation*}
\frac{\lbrack n]_{p_{n},q_{n}}[n-1]_{p_{n},q_{n}}}{[n+1]_{p_{n},q_{n}}^{2}}=%
\frac{1}{q_{n}^{3}}\left\{ 1-p_{n}^{n}\left( 2+\frac{q_{n}}{p_{n}}\right)
\frac{1}{[n+1]_{p_{n},q_{n}}}+(p_{n}^{n})^{2}\left( 1+\frac{q_{n}}{p_{n}}%
\right) \frac{1}{[n+1]_{p_{n},q_{n}}^{2}}\right\} ,
\end{equation*}%
and
\begin{equation*}
\frac{\lbrack n]_{p_{n},q_{n}}}{[n+1]_{p_{n},q_{n}}^{2}}=\frac{1}{q_{n}}%
\left( \frac{1}{[n+1]_{p_{n},q_{n}}}-p_{n}^{n}\frac{1}{%
[n+1]_{p_{n},q_{n}}^{2}}\right) .
\end{equation*}%
Thus, we have\newline
$\parallel L_{n}^{p_{n},q_{n}}\left( \left( \frac{t}{1+t}\right)
^{2};x\right) -\left( \frac{x}{1+x}\right) ^{2}\parallel _{C_{B}}$ \\ \newline
$\leq \frac{p_{n}}{q_{n}}\left\{ 1-p_{n}^{n}\left( 2+\frac{q_{n}}{p_{n}}%
\right) \frac{1}{[n+1]_{p_{n},q_{n}}}+(p_{n}^{n})^{2}\left( 1+\frac{q_{n}}{%
p_{n}}\right) \frac{1}{[n+1]_{p_{n},q_{n}}^{2}}-1\right\}$\\ %
$+p_{n}^{n}\frac{%
p_{n}}{q_{n}}\left( \frac{1}{[n+1]_{p_{n},q_{n}}}-p_{n}^{n}\frac{1}{%
[n+1]_{p_{n},q_{n}}^{2}}\right) .$\\ \newline
This implies that the condition \eqref{nas6} holds for $\nu =2$ and the proof is completed by Theorem \ref{main}.
\end{proof}

\textbf{\ Rate of Convergence.}\\ \newline

In this section, we calculate the rate of convergence of operators \eqref{nas3} by
means of modulus of continuity and Lipschitz type maximal functions.

\parindent=8mmThe modulus of continuity for $f\in H_{\omega }$ is defined by
\begin{equation*}
\widetilde{\omega }(f;\delta )=\sum_{\substack{ \mid \frac{t}{1+t}-\frac{x}{%
1+x}\mid \leq \delta ,  \\ x,t\geq 0}}\mid f(t)-f(x)\mid
\end{equation*}%
where $\widetilde{\omega }(f;\delta )$ satisfies the following conditions.
For all $f\in H_{\omega }(\mathbb{R}_{+})$

\begin{enumerate}
\item $\lim_{\delta \to 0}\widetilde{\omega}(f; \delta)=0$

\item $\mid f(t)-f(x) \mid \leq \widetilde{\omega}(f; \delta)
\left( \frac{\mid \frac{t}{1+t}-\frac{x}{1+x}\mid}{\delta}+1 \right)$
\end{enumerate}

\begin{theorem}\label{main9}
Let $p=p_{n},~~q=q_{n}$ satisfy \eqref{nas5}, for $0<q_{n}<p_{n}\leq 1$ and if $%
L_{n}^{p_{n},q_{n}}$ is defined by \eqref{nas3}. Then for each $x\geq 0$ and for
any function $f\in H_{\omega }$, we have
\begin{equation*}
\mid L_{n}^{p_{n},q_{n}}(f;x)-f\mid \leq 2\widetilde{\omega }(f;\sqrt{\delta
_{n}(x)}),
\end{equation*}%
where
\begin{equation*}
\delta _{n}(x)=\frac{x^{2}}{(1+x)^{2}}\left( \frac{%
p_{n}q_{n}^{2}[n]_{p_{n},q_{n}}[n-1]_{p_{n},q_{n}}}{[n+1]_{p_{n},q_{n}}^{2}}%
\frac{1+x}{p_{n}+q_{n}x}-2\frac{p_{n}[n]_{p_{n},q_{n}}}{[n+1]_{p_{n},q_{n}}}%
+1\right) +\frac{p_{n}^{n+1}[n]_{p_{n},q_{n}}}{[n+1]_{p_{n},q_{n}}^{2}}\frac{%
x}{1+x}.
\end{equation*}
\end{theorem}

\begin{proof}
\begin{eqnarray*}
\mid L_n^{p_n,q_n}(f;x)-f \mid &\leq & L_n^{p_n,q_n}\left(\mid f(t)-f(x)
\mid;x \right) \\
&\leq & \widetilde{\omega}(f; \delta) \left\{1+\frac{1}{\delta}
L_n^{p_n,q_n} \left( \bigg{|} \frac{t}{1+t}-\frac{x}{1+x}\big{|};x
\right)\right\}.
\end{eqnarray*}
Now by using the Cauchy-Schwarz inequality, we have
\begin{eqnarray*}
\mid L_n^{p_n,q_n}(f;x)-f \mid &\leq & \widetilde{\omega}(f; \delta_n)
\left\{1+\frac{1}{\delta_n} \left[ \left( L_n^{p_n,q_n} \left(\frac{t}{1+t}-%
\frac{x}{1+x}\right)^2;x \right) \right]^{\frac{1}{2}} \left(
L_n^{p_n,q_n}(1;x)\right)^{\frac{1}{2}}\right\}
\end{eqnarray*}
$\leq \widetilde{\omega}(f; \delta_n) \left\{ 1+ \frac{1}{\delta_n} \left[
\frac{x^2}{(1+x)^2}\left( \frac{p_nq_n^2[n]_{p_n,q_n}[n-1]_{p_n,q_n}}{%
[n+1]_{p_n,q_n}^2} \frac{1+x}{p_n+q_nx} -2\frac{p_n[n]_{p_n,q_n}}{%
[n+1]_{p_n,q_n}}+1\right)+\frac{p_n^{n+1}[n]_{p_n,q_n}}{[n+1]_{p_n,q_n}^2}%
\frac{x}{1+x}\right]^{\frac{1}{2}}\right\}$.\newline

\parindent=8mmThis completes the proof.
\end{proof}
\parindent=8mmNow we will give an estimate concerning the rate of
convergence by means of Lipschitz type maximal functions. In \cite{aral1},
the Lipschitz type maximal function space on $E\subset \mathbb{R}_{+}$ is
defined as
\begin{align}\label{nas7}
\widetilde{W}_{\alpha ,E}=\{f:\sup (1+x)^{\alpha }\widetilde{f}_{\alpha
}(x)\leq M\frac{1}{(1+y)^{\alpha }}:x\leq 0,~\mbox{and}~y\in
E\}
\end{align}%
where $f$ is bounded and continuous function on $\mathbb{R}_{+}$, $M$ is a
positive constant, $0<\alpha \leq 1$.

In \cite{lenz}, B.Lenze introduced a Lipschitz type maximal function $%
f_{\alpha }$ as follows:
\begin{align}\label{nas8}
f_{\alpha }(x,t)=\sum_{\substack{ t>0  \\ t\neq x}}\frac{\mid f(t)-f(x)\mid
}{\mid x-t\mid ^{\alpha }}.
\end{align}%
We denote by $d(x,E)$ the distance between $x$ and $E$, that is
\begin{equation*}
d(x,E)=\inf \{\mid x-y\mid ;y\in E\}.
\end{equation*}

\begin{theorem}\label{main8}
For all $f \in \widetilde{W}_{\alpha, E},$ we have
\begin{align}\label{nas9}
\mid L_n^{p_n,q_n}(f;x) -f(x) \mid \leq M\left( \delta_n^{\frac{\alpha}{2}%
}(x)+2 \left( d(x,E)\right)^\alpha\right)
\end{align}
where $\delta_n(x)$ is defined as in Theorem \ref{main9}.
\end{theorem}

\begin{proof}
Let $\overline{E}$ denote the closure of the set $E$. Then there exits a $%
x_{0}\in \overline{E}$ such that $\mid x-x_{0}\mid =d(x,E)$, where $x\in
\mathbb{R}_{+}$. Thus we can write
\begin{equation*}
\mid f-f(x)\mid \leq \mid f-f(x_{0})\mid +\mid f(x_{0})-f(x)\mid .
\end{equation*}%
Since $L_{n}^{p_{n},q_{n}}$ is a positive linear operator, $f\in \widetilde{W%
}_{\alpha ,E}$ and by using the previous inequality, we have
\begin{equation*}
\mid L_{n}^{p_{n},q_{n}}(f;x)-f(x)\mid \leq \mid L_{n}^{p_{n},q_{n}}(\mid
f-f(x_{0})\mid ;x)+\mid f(x_{0})-f(x)\mid L_{n}^{p_{n},q_{n}}(1;x)
\end{equation*}%
\begin{equation*}
\leq M\left( L_{n}^{p_{n},q_{n}}\left( \bigg{|}\frac{t}{1+t}-\frac{x_{0}}{%
1+x_{0}}\bigg{|}^{\alpha };x\right) +\frac{\mid x-x_{0}\mid ^{\alpha }}{%
(1+x)^{\alpha }(1+x_{0})^{\alpha }}L_{n}^{p_{n},q_{n}}(1;x)\right) .
\end{equation*}%
Since $(a+b)^{\alpha }\leq a^{\alpha }+b^{\alpha }$, which consequently
imply
\begin{equation*}
L_{n}^{p_{n},q_{n}}\left( \bigg{|}\frac{t}{1+t}-\frac{x_{0}}{1+x_{0}}\bigg{|}%
^{\alpha };x\right) \leq L_{n}^{p_{n},q_{n}}\left( \bigg{|}\frac{t}{1+t}-%
\frac{x}{1+x}\bigg{|}^{\alpha };x\right) +L_{n}^{p_{n},q_{n}}\left( \bigg{|}%
\frac{x}{1+x}-\frac{x_{0}}{1+x_{0}}\bigg{|}^{\alpha };x\right)
\end{equation*}

\begin{equation*}
L_{n}^{p_{n},q_{n}}\left( \bigg{|}\frac{t}{1+t}-\frac{x_{0}}{1+x_{0}}\bigg{|}%
^{\alpha };x\right) \leq L_{n}^{p_{n},q_{n}}\left( \bigg{|}\frac{t}{1+t}-%
\frac{x}{1+x}\bigg{|}^{\alpha };x\right) +\frac{\mid x-x_{0}\mid ^{\alpha }}{%
(1+x)^{\alpha }(1+x_{0})^{\alpha }}L_{n}^{p_{n},q_{n}}(1;x).
\end{equation*}%
By using the Hölder inequality with $p=\frac{2}{\alpha }$ and $q=\frac{2}{%
2-\alpha }$, we have\\ \newline
$L_{n}^{p_{n},q_{n}}\left( \bigg{|}\frac{t}{1+t}-\frac{x_{0}}{1+x_{0}}%
\bigg{|}^{\alpha };x\right) \leq L_{n}^{p_{n},q_{n}}\left( \left( \frac{t}{%
1+t}-\frac{x}{1+x}\right) ^{2};x\right) ^{\frac{\alpha }{2}%
}(L_{n}^{p_{n},q_{n}}(1;x))^{\frac{2-\alpha }{2}}$
\begin{equation*}
+\frac{\mid x-x_{0}\mid
^{\alpha }}{(1+x)^{\alpha }(1+x_{0})^{\alpha }}L_{n}^{p_{n},q_{n}}(1;x)
\end{equation*}

\begin{equation*}
=\delta _{n}^{\frac{\alpha }{2}}(x)+\frac{\mid x-x_{0}\mid ^{\alpha }}{%
(1+x)^{\alpha }(1+x_{0})^{\alpha }}.
\end{equation*}

\parindent=8mmThis completes the proof.
\end{proof}
\begin{corollary}\label{main12}
If we take $E= \mathbb{R}_+$ as a particular case of Theorem \ref{main8}, then for
all $f \in \widetilde{W}_{\alpha, \mathbb{R}_+}$, we have
\begin{equation*}
\mid L_n^{p_n,q_n}(f;x) -f(x) \mid \leq M \delta_n^{\frac{\alpha}{2}}(x),
\end{equation*}
\end{corollary}

where $\delta_n(x)$ is defined in Theorem \ref{main9}.
\begin{theorem}
If $x\in (0,\infty )\backslash \left\{ p^{n-k+1}\frac{[k]_{p,q}}{%
[n-k+1]_{p,q}q^{k}}\bigg{|}k=0,1,2,\cdots ,n\right\} $, then

$L_n^{p,q}(f;x)-f\left(\frac{px}{q}\right)=-\frac{x^{n+1}}{\ell_n^{p,q}(x)} %
\left[\frac{px}{q};\frac{p[n]_{p,q}}{q^n};f\right] pq^{\frac{n(n-1)}{2}-1}$
\begin{align}\label{nas10}
+\frac{x}{\ell_n^{p,q}(x)}\sum_{k=0}^{n-1} \left[\frac{px}{q};p^{n-k+1}\frac{%
[k]_{p,q}}{[n-k+1]_{p,q}q^{k} };f\right] \frac{1}{[n-k]_{p,q}} p^{\frac{%
(n-k)(n-k+1)}{2}+1}q^{\frac{k(k-3)}{2}-2} \left[
\begin{array}{c}
n \\
k%
\end{array}%
\right] _{p,q} x^{k}.
\end{align}
\end{theorem}

\begin{proof}
By using \eqref{nas3}, we have\\ \newline
$L_n^{p,q}(f;x)-f\left(\frac{px}{q}\right)=\frac{1}{\ell_n^{p,q}(x)}%
\sum_{k=0}^n \left[f \left( \frac{ p^{n-k+1}[k]_{p,q}}{[n-k+1]_{p,q}q^k }%
\right) -f\left(\frac{px}{q}\right)\right] p^{\frac{(n-k)(n-k-1)}{2}}q^{%
\frac{k(k-1)}{2}} \left[
\begin{array}{c}
n \\
k%
\end{array}%
\right] _{p,q} x^k $\\
$=-\frac{1}{\ell_n^{p,q}(x)}\sum_{k=0}^n \left(\frac{px}{q}- \frac{
p^{n-k+1}[k]_{p,q}}{[n-k+1]_{p,q}q^k }\right) \left[\frac{px}{q}; \frac{%
p^{n-k+1}[k]_{p,q}}{[n-k+1]_{p,q}q^k };f\right] p^{\frac{(n-k)(n-k-1)}{2}}q^{%
\frac{k(k-1)}{2}} \left[
\begin{array}{c}
n \\
k%
\end{array}%
\right] _{p,q} x^k$\\ \newline
By using $\frac{[k]_{p,q}}{[n-k+1]_{p,q}}\left[
\begin{array}{c}
n \\
k%
\end{array}%
\right] _{p,q}=\left[
\begin{array}{c}
n \\
k-1%
\end{array}%
\right] _{p,q}$, we have

\begin{equation*}
L_n^{p,q}(f;x)-f\left(\frac{px}{q}\right)=-\frac{x}{\ell_n^{p,q}(x)}%
\sum_{k=0}^n \left[\frac{px}{q};\frac{ p^{n-k+1}[k]_{p,q}}{[n-k+1]_{p,q}q^k }%
;f\right] p^{\frac{(n-k)(n-k-1)}{2}+1}q^{\frac{k(k-1)}{2}-1} \left[
\begin{array}{c}
n \\
k%
\end{array}%
\right] _{p,q} x^k.
\end{equation*}

\begin{equation*}
+\frac{1}{\ell_n^{p,q}(x)}\sum_{k=1}^n \left[\frac{px}{q};\frac{ p^{n-k+1}
[k]_{p,q}}{[n-k+1]_{p,q}q^k };f\right] p^{\frac{(n-k)(n-k-1)}{2}-(k-n-1)}q^{%
\frac{k(k-1)}{2}-k} \left[
\begin{array}{c}
n \\
k-1%
\end{array}%
\right] _{p,q} x^k
\end{equation*}

\begin{equation*}
=-\frac{x}{\ell_n^{p,q}(x)}\sum_{k=0}^n \left[\frac{px}{q};\frac{ p^{n-k+1}
[k]_{p,q}}{[n-k+1]_{p,q}q^k };f\right] p^{\frac{(n-k)(n-k-1)}{2}+1}q^{\frac{%
k(k-1)}{2}-1} \left[
\begin{array}{c}
n \\
k%
\end{array}%
\right] _{p,q} x^k.
\end{equation*}

\begin{equation*}
+\frac{x}{\ell_n^{p,q}(x)}\sum_{k=0}^{n-1} \left[\frac{px}{q};\frac{
p^{n-k}[k+1]_{p,q}}{[n-k]_{p,q}q^{k+1} };f\right] p^{\frac{(n-k-1)(n-k-2)}{2}%
-(k-n)}q^{\frac{k(k+1)}{2}-(k+1)} \left[
\begin{array}{c}
n \\
k%
\end{array}%
\right] _{p,q} x^{k}
\end{equation*}

$=-\frac{x^{n+1}}{\ell_n^{p,q}(x)} \left[\frac{px}{q};\frac{p[n]_{p,q}}{q^n}%
;f\right] pq^{\frac{n(n-1)}{2}-1} $\\ \newline
$+\frac{x}{\ell_n^{p,q}(x)}\sum_{k=0}^{n-1}\left \{ \left[\frac{px}{q};\frac{
p^{n-k}[k+1]_{p,q}}{[n-k]_{p,q}q^{k+1} };f\right] -\left[\frac{px}{q};\frac{
p^{n-k+1} [k]_{p,q}}{[n-k+1]_{p,q}q^{k} };f\right] \right \} p^{\frac{%
(n-k)(n-k-1)}{2}+1}q^{\frac{k(k-1)}{2}-1} \left[
\begin{array}{c}
n \\
k%
\end{array}%
\right] _{p,q} x^{k}.$\\ \newline

Now by using the result\\ \newline
$\left[\frac{px}{q};\frac{ p^{n-k}[k+1]_{p,q}}{[n-k]_{p,q}q^{k+1} };f\right]
-\left[\frac{px}{q};\frac{ p^{n-k+1} [k]_{p,q}}{[n-k+1]_{p,q}q^{k} };f\right]
$
\begin{equation*}
= \left( \frac{ p^{n-k}[k+1]_{p,q}}{[n-k]_{p,q}q^{k+1} }-\frac{ p^{n-k+1}
[k]_{p,q}}{[n-k+1]_{p,q}q^{k} }\right) f \left[ \frac{px}{q};\frac{
p^{n-k+1} [k]_{p,q}}{[n-k+1]_{p,q}q^{k} };\frac{ p^{n-k} [k+1]_{p,q}}{%
[n-k]_{p,q}q^{k+1} };f \right]
\end{equation*}
and
\begin{equation*}
\frac{ p^{n-k} [k+1]_{p,q}}{[n-k]_{p,q}q^{k+1} }-\frac{ p^{n-k+1} [k]_{p,q}}{%
[n-k+1]_{p,q}q^{k} }=[n+1]_{p,q},
\end{equation*}

we have\newline
\newline

$L_n^{p,q}(f;x)-f\left(\frac{px}{q}\right)=-\frac{x^{n+1}}{\ell_n^{p,q}(x)} %
\left[\frac{px}{q};\frac{p[n]_{p,q}}{q^n};f\right] pq^{\frac{n(n-1)}{2}-1} $%
\newline
$+\frac{x}{\ell_n^{p,q}(x)}\sum_{k=0}^{n-1}\left \{ \left[\frac{px}{q};\frac{
p^{n-k+1} [k]_{p,q}}{[n-k+1]_{p,q}q^{k} };f\right] \frac{ p^{n-k}[n+1]_{p,q}%
}{[n-k]_{p,q}[n-k+1]_{p,q}q^{k+1} } \right\} p^{\frac{(n-k)(n-k-1)}{2}+1}q^{%
\frac{k(k-1)}{2}-1} \left[
\begin{array}{c}
n \\
k%
\end{array}%
\right] _{p,q} x^{k}.$\newline

\parindent=8mmThis completes the proof.
\end{proof}
\textbf{ \ Some Generalization of $L_n^{p,q}$.}\\

In this section, we present some generalization of the operators $L_n^{p,q}$
based on $(p,q)$-integers similar to work done in \cite{n1, aral1}.

We consider a sequence of linear positive operators based on $(p,q)$%
-integers as follows:
\begin{align}\label{nas11}
L_{n }^{(p,q),\gamma}(f;x)=\frac{1}{\ell _{n}^{p,q}(x)}\sum_{k=0}^{n}f\left(
\frac{p^{n-k+1}[k]_{p,q}+\gamma }{b_{n,k}}\right) p^{\frac{(n-k)(n-k-1)}{2}%
}q^{\frac{k(k-1)}{2}}\left[
\begin{array}{c}
n \\
k%
\end{array}%
\right] _{p,q}x^{k},~~~~~~(\gamma \in \mathbb{R})
\end{align}%
where $b_{n,k}$ satisfies the following conditions:
\begin{equation*}
p^{n-k+1}[k]_{p,q}+b_{n,k}=c_{n}~~~~~~\mbox{and}~~~~\frac{[n]_{p,q}}{c_{n}}%
\rightarrow 1,~~~~~~\mbox{for}~~~~n\rightarrow \infty .
\end{equation*}%
It is easy to check that if $b_{n,k}=q^{k}[n-k+1]_{p,q}+\beta $ for any $n,k$
and $0<q<p\leq 1$, then $c_{n}=[n+1]_{p,q}+\beta $. If we choose $p=1$, then
operators reduce to generalization of $q$-BBH opeartors defined in \cite%
{aral1}, and which turn out to be D.D. Stancu-type generalization of
Bleimann, Butzer, and Hahn operators based on $q$-integers \cite{n2}. If we
choose $\gamma =0,~~~q=1$ as in \cite{aral1} for $p=1$, then the operators
become the special case of the Balázs-type generalization of the $q$-BBH
operators \cite{aral1} given in \cite{n1}.

\begin{theorem}
Let $p=p_n,~~q=q_n$ satisfying \eqref{nas5}, for $0<q_n<p_n\leq 1$ and if $%
L_n^{(p_n,q_n), \gamma}$ is defined by \eqref{nas11}, then for any function $f \in
\widetilde{W}_{\alpha, [0,\infty)}$, we have\\ \newline
$\lim_{n}\parallel L_{n }^{(p_{n},q_{n}),\gamma}(f;x)-f(x)\parallel _{C_{B}}\leq 3M $\\%
\newline
$\times \max \left\{ \left( \frac{[n]_{p_{n},q_{n}}}{c_{n}+\gamma }\right)
^{\alpha }\left( \frac{\gamma }{[n]_{p_{n},q_{n}}}\right) ^{\alpha },%
\bigg{|}1-\frac{[n+1]_{p_{n},q_{n}}}{c_{n}+\gamma }\bigg{|}^{\alpha }\left(
\frac{p_{n}[n]_{p_{n},q_{n}}}{[n+1]_{p_{n},q_{n}}}\right) ^{\alpha },1-2%
\frac{p_{n}[n]_{p_{n},q_{n}}}{[n+1]_{p_{n},q_{n}}}+\frac{%
q_{n}[n]_{p_{n},q_{n}}[n-1]_{p_{n},q_{n}}}{[n+1]_{p_{n},q_{n}}^{2}}\right\}.$
\end{theorem}
\begin{proof}

Using \eqref{nas3} and \eqref{nas11}, we have \\ \newline
$\mid L_{n }^{(p,q),\gamma}(f;x)-f(x)\mid $ \\ \newline
$\leq \frac{1}{\ell _{n}^{p_{n},q_{n}}(x)}\sum_{k=0}^{n}\bigg{|}f\left(
\frac{p_{n}^{n-k+1}[k]_{p_{n},q_{n}}+\gamma }{b_{n,k}}\right) -f\left( \frac{%
p_{n}^{n-k+1}[k]_{p_{n},q_{n}}}{\gamma +b_{n,k}}\right) \bigg{|}p_{n}^{\frac{%
(n-k)(n-k-1)}{2}}q_{n}^{\frac{k(k-1)}{2}}\left[
\begin{array}{c}
n \\
k%
\end{array}%
\right] _{p_{n},q_{n}}x^{k}$
\begin{equation*}
+ \frac{1}{\ell_n^{p_n,q_n}(x)}\sum_{k=0}^n \bigg{|} f \left(\frac{%
p_n^{n-k+1}[k]_{p_n,q_n}}{\gamma+ b_{n,k}}\right) -f\left(\frac{%
p_n^{n-k+1}[k]_{p_n,q_n}}{[n-k+1]_{p_n,q_n}q_n^{k}}\right)\bigg{|} p_n^{%
\frac{(n-k)(n-k-1)}{2}}q_n^{\frac{k(k-1)}{2}} \left[
\begin{array}{c}
n \\
k%
\end{array}%
\right] _{p_n,q_n} x^k
\end{equation*}
$+ \mid L_n^{p_n,q_n}(f;x) -f(x) \mid.$\\ \newline
Since $f\in \widetilde{W}_{\alpha ,[0,\infty )}$ and by using the Corollary
\ref{main12}, we can write\\ \newline
$\mid L_{n }^{(p,q),\gamma}(f;x)-f(x)\mid $\\ \newline
$\leq \frac{M}{\ell _{n}^{p_{n},q_{n}}(x)}\sum_{k=0}^{n}\bigg{|}\frac{%
p_{n}^{n-k+1}[k]_{p_{n},q_{n}}+\gamma }{p_{n}^{n-k+1}[k]_{p_{n},q_{n}}+%
\gamma +b_{n,k}}-\frac{p_{n}^{n-k+1}[k]_{p_{n},q_{n}}}{\gamma
+p_{n}^{n-k+1}[k]_{p_{n},q_{n}}+b_{n,k}}\bigg{|}^{\alpha }p_{n}^{\frac{%
(n-k)(n-k-1)}{2}}q_{n}^{\frac{k(k-1)}{2}}\left[
\begin{array}{c}
n \\
k%
\end{array}%
\right] _{p_{n},q_{n}}x^{k}$\\
$+\frac{M}{\ell _{n}^{p_{n},q_{n}}(x)}\sum_{k=0}^{n}\bigg{|}\frac{%
p_{n}^{n-k+1}[k]_{p_{n},q_{n}}}{p_{n}^{n-k+1}[k]_{p_{n},q_{n}}+\gamma
+b_{n,k}}-\frac{p_{n}^{n-k+1}[k]_{p_{n},q_{n}}}{%
p_{n}^{n-k+1}[k]_{p_{n},q_{n}}+[n-k+1]_{p_{n},q_{n}}q_{n}^{k}}\bigg{|}$\\
$\times p_{n}^{%
\frac{(n-k)(n-k-1)}{2}}q_{n}^{\frac{k(k-1)}{2}}\left[
\begin{array}{c}
n \\
k%
\end{array}%
\right] _{p_{n},q_{n}}x^{k} +M\delta _{n}^{\frac{\alpha }{2}}(x).$\newline

This implies that \\ \newline
$\mid L_{n }^{(p,q),\gamma}(f;x)-f(x)\mid \leq M\left( \frac{[n]_{p_{n},q_{n}}%
}{c_{n}+\gamma }\right) ^{\alpha }\left( \frac{\gamma }{[n]_{p_{n},q_{n}}}%
\right) ^{\alpha }$
\begin{equation*}
+\frac{M}{\ell _{n}^{p_{n},q_{n}}(x)}\bigg{|}1-\frac{[n+1]_{p_{n},q_{n}}}{%
c_{n}+\gamma }\bigg{|}^{\alpha }\sum_{k=0}^{n}\left( \frac{%
p_{n}^{n-k+1}[k]_{p_{n},q_{n}}}{[n+1]_{p_{n},q_{n}}}\right) ^{\alpha }p_{n}^{%
\frac{(n-k)(n-k-1)}{2}}q_{n}^{\frac{k(k-1)}{2}}\left[
\begin{array}{c}
n \\
k%
\end{array}%
\right] _{p_{n},q_{n}}x^{k}+M\delta _{n}^{\frac{\alpha }{2}}(x)
\end{equation*}

\begin{equation*}
=M\left( \frac{[n]_{p_{n},q_{n}}}{c_{n}+\gamma }\right) ^{\alpha }\left(
\frac{\gamma }{[n]_{p_{n},q_{n}}}\right) ^{\alpha }+M\bigg{|}1-\frac{%
[n+1]_{p_{n},q_{n}}}{c_{n}+\gamma }\bigg{|}^{\alpha
}L_{n}^{p_{n},q_{n}}\left( \left( \frac{t}{1+t}\right) ^{\alpha };x\right)
+M\delta _{n}^{\frac{\alpha }{2}}(x).
\end{equation*}%
Using the Hölder inequality for $p=\frac{1}{\alpha },~~~q=\frac{1}{1-\alpha }
$, we get\\ \newline
$\mid L_{n }^{(p,q),\gamma}(f;x)-f(x)\mid $\\ \newline
$\leq M\left( \frac{[n]_{p_{n},q_{n}}}{c_{n}+\gamma }\right) ^{\alpha }\left(
\frac{\gamma }{[n]_{p_{n},q_{n}}}\right) ^{\alpha }+M\bigg{|}1-\frac{%
[n+1]_{p_{n},q_{n}}}{c_{n}+\gamma }\bigg{|}^{\alpha
}L_{n}^{p_{n},q_{n}}\left( \frac{t}{1+t};x\right) ^{\alpha }\left(
L_{n}^{p_{n},q_{n}}(1;x)\right) ^{1-\alpha }+M\delta _{n}^{\frac{\alpha }{2}%
}(x)$\\
$\leq M\left( \frac{[n]_{p_{n},q_{n}}}{c_{n}+\gamma }\right) ^{\alpha }\left(
\frac{\gamma }{[n]_{p_{n},q_{n}}}\right) ^{\alpha }+M\bigg{|}1-\frac{%
[n+1]_{p_{n},q_{n}}}{c_{n}+\gamma }\bigg{|}^{\alpha }\left( \frac{%
p_{n}[n]_{p_{n},q_{n}}}{[n+1]_{p_{n},q_{n}}}\frac{x}{1+x}\right) ^{\alpha
}+M\delta _{n}^{\frac{\alpha }{2}}(x).$\\
This completes the proof.
\end{proof}

{\bf Acknowledgement.} Acknowledgements could be placed at the end
of the text but precede the references.

\bibliographystyle{amsplain}

\end{document}